\documentclass[12pt, twoside, leqno]{article}

\usepackage{amsmath,amsthm}
\usepackage{amssymb}

\usepackage{enumitem}

\usepackage{graphicx}

\usepackage[T1]{fontenc}

\usepackage[style=numeric, backend=bibtex]{biblatex}
\usepackage[title]{appendix}

\newtheorem{thm}{Theorem}[section]
\newtheorem{prop}[thm]{Proposition}
\newtheorem{lem}[thm]{Lemma}
\newtheorem{cor}[thm]{Corollary}
\newtheorem{conj}[thm]{Conjecture}

\theoremstyle{definition}
\newtheorem{defn}[thm]{Definition}

\theoremstyle{remark}
\newtheorem{remark}[thm]{Remark}

\numberwithin{equation}{section}

\usepackage{enumerate}
\usepackage{verbatim}
\usepackage{enumitem}

\newcommand{\F}{\mathbb{F}}
\newcommand{\Z}{\mathbb{Z}}
\newcommand{\Q}{\mathbb{Q}}

\newcommand{\Fn}{\mathbb{F}_{q^n}}
\newcommand{\Aut}{\operatorname{Aut}}

\newcommand{\End}{\operatorname{End}}

\newcommand{\Ker}{\operatorname{Ker}}

\newcommand{\rvline}{\hspace*{-\arraycolsep}\vline\hspace*{-\arraycolsep}}

\numberwithin{equation}{section}

\begin{document}


\baselineskip=17pt


\title{Isogeny Classes of Split Almost Ordinary Abelian Surfaces over Finite Fields}

\author{Yu Fu}

\date{}
\maketitle
\begin{abstract}
Let $A=E \times E_{ss}$ be a principally polarized almost ordinary split abelian surface over a finite field $\mathbb{F}_{q}$. We give asymptotic upper and lower bounds on the number of principally polarized abelian surfaces over $\mathbb{F}_{q^{n}}$ that are $\overline{\mathbb{F}}_{q}$-isogenous to $A$ up to isomorphism, which is a refinement of the results in \cite{AH17}.
\end{abstract}

\renewcommand{\thefootnote}{}



\renewcommand{\thefootnote}{\arabic{footnote}}
\setcounter{footnote}{0}

\section{Introduction}
Many fundamental problems on Shimura varieties pertain to the behavior of isogeny classes, for example, the Hecke orbit conjecture and specific questions related to unlikely intersections. In \cite[Theorem 4.1]{ST18}, Shankar and Tsimerman proved an asymptotic formula for the size of the isogeny class of ordinary elliptic curves over finite fields. As an application, they proved the existence of a hypersurface in the moduli space $X(1)^{270}$, which intersects every isogeny class. 

A few common strategies exist to obtain asymptotic formulas for the size of isogeny classes of abelian varieties over finite fields. In particular, when the abelian variety is ordinary and simple, the inspiring work of Deligne \cite{De69} explicitly classified such abelian varieties over finite fields. Using the classification, one can get bounds for the isogeny classes of ordinary abelian varieties, for example, \cite[Theorem 3.3]{ST18}. A handful of studies in this flavor have been performed in more general settings. For example, one may refer to \cite{OS20} when the abelian variety is almost-ordinary and geometrically simple and to \cite{BF23} for a setting of Hilbert modular varieties. All of the results above depend on the existence of canonical lifting and classification of abelian varieties over finite fields.
 
A second way of doing this is to interpret isogeny classes in terms of orbital integrals. For example, in \cite{AC02}, Achter and Cunningham proved an explicit formula for the size of the isogeny class of a Hilbert-Blumenthal abelian variety over a finite field. They express the size of the isogeny class as a product of local orbital integrals on $GL(2)$ and then evaluate all the relevant orbital integrals. See also \cite{AW14} where Achter and Williams proved that for a particular class of simple, ordinary abelian surfaces over $\F_{q}$ given by a $q$-weil polynomial $f$, the number of principally polarized abelian surfaces over $\F_{q}$ with Weil polynomial $f$ could be calculated by an infinite product of local factors which can be calculated by method of orbital integrals.

 Throughout this article, let $(A, \lambda_{A})$ be a principally polarized non-simple abelian surface defined over $\F_{q}$, with the polarization given by $\lambda_{A}$. Moreover, assume that $A$ has the form $A=E \times E_{ss}$, where $E$ is an ordinary elliptic curve and $E_{ss}$ is a supersingular elliptic curve. The endomorphism algebra $\End^{\circ}(A)$ is non-commutative, and there is no canonical lifting of $A$. Therefore, we cannot interpret the question as estimating the size of class groups by using the classification of abelian varieties over finite fields. Instead, we present a proof of the lower bound of the isogeny class of $A$ defined over $\F_{q}$ and describe how this cardinality is affected by the base change to finite extensions of $\mathbb{F}_{q}$ by using group-theoretical methods. As pointed out by the referee, there is another way to prove the fact, especially to approach the case when $n$ is odd and for the upper bound, using a different perspective based on the work of Frey and Kani \cite{FK91} and Achter and Howe \cite{AH17}. See Section 3.4.2 and Appendix A for details about this method.

 Before introducing the main theorem, we introduce some notation. Let $I(q^{n},A)$ be the set of principally polarized abelian varieties defined over $\mathbb{F}_{q^{n}}$ that are isogenous to $A$ over $\overline{\mathbb{F}}_{q}$. Let $N(q^{n},A)$ denote the cardinality of $I(q^{n},A)$. 
 By interpreting the question as a classification of finite subgroup schemes, we obtain a lower bound on the number of principally polarized abelian varieties over $\mathbb{F}_{q^{n}}$ that is isogenous to $A$ over $\overline{\mathbb{F}}_{q}$. Our main result is the following.
 
\begin{thm}\label{main}
Let $(A,\lambda_{A})$ be a principally polarized abelian variety over $\F_{q}$ such that $A=E\times E_{ss}$. Let $K$ be the quadratic number field such that $K=\End^{\circ}(E)$. Suppose $n$ satisfy one of the following conditions
\begin{itemize}
    \item $n$ is even and $n$ is an integer coprime to the discriminant of $\mathcal{O}_K$. 
    \item  $n$ is odd and $-1$ is a quadratic residue modulo the trace of Frobenius of $E.$
\end{itemize}
 Then $$N(q^{n},A) \ge q^{n + o(n)}.$$
 Moreover, for all positive integer $n$, we have  $$N(q^{n},A) \le q^{n + o(n)}.$$
\end{thm}

We note that the lower bound in \cite[Theorem 1.2]{AH17} came from estimating the number of \textit{ordinary split non-isotypic surfaces}, which are abelian surfaces isogenous to a product $E_1 \times E_2$ of two ordinary elliptic curves, with $E_1$ and $E_2$ lying in two different isogeny classes, and the author did not specify the counting problem in a fixed isogeny class throughout the paper. Therefore, our theorem, in particular the lower bound, is a refinement of their results.

In addition, we provide an approach to count the size of isogeny classes of ordinary elliptic curves over finite fields, the upper bound of which is known by Lenstra \cite[Proposition 1.19]{Len} and Shankar-Tsimerman \cite[Theorem 3.3]{ST18}, using a different perspective. Due to \cite[Lemma 5.4]{BF23}, we can generalize \cite[Theorem 3.3]{ST18} to all but finitely many positive integers $n$.

\begin{thm}\label{ordinarythm}
	Let $E$ be an ordinary elliptic curve defined over $\F_{q}$. For all but finitely many positive integers $n$, we have
	$$N(q^n, E)= (q^n)^{1/2 + o(1)}.$$
\end{thm} 

\vspace{1em}
There is a general conjecture regarding the size of the isogeny class of abelian varieties over finite fields. Let $N(W)$ be the open Newton stratum of $\mathcal{A}_{g}$ consisting of all abelian varieties whose Newton polygon is $W$ and let $A$ be a principally polarized abelian variety in $\mathcal{A}_{g}$. Recall that the \textit{central leaf} through $A$ consists of all abelian varieties in $N(W)$ whose $p$-divisible group is isomorphic to $A[p^{\infty}]$. The \textit{isogeny leaf} through $A$ is a maximal irreducible subscheme of $\mathcal{A}_{g}$ consisting of abelian varieties $A^{\prime}$ in $N(W)$ such that $A^{\prime}$ is isogenous to $A$ through an isogeny whose kernel is an iteration extension of the group scheme $\alpha_{p}$. Let $\dim(CL)$ be the dimension of the central leaf through $A$ and let $\dim(IL)$ be the dimension of the isogeny leaf through $A$.

\begin{conj}\label{conjecture}
	We have $$N(q^{n},A)=q^{n(\frac{\dim(CL)}{2}+ \dim(IL)) + o(1)}.$$ 
\end{conj} 

All the previous results we state above satisfy the Conjecture \ref{conjecture}. When $A$ is a non-simple abelian surface, it is easy to see that the dimension of the central leaf through $A$ is $2$, by the formula of lattice-point count by Shankar and Tsimerman \cite[Section 5.2]{ST18}. The dimension of the isogeny leaf through $A$ is $0$. Therefore the conjecture is true in this case.

\section{The lower bound: from the group theoretical perspective}
\subsection{The Isogeny Classes and Maximal Isotropic Subgroups}

A classical way to construct abelian varieties isogenous to a fixed abelian variety $A$ is to take quotients of $A$ by finite subgroup schemes. A theorem of Mumford \cite[II.7 Theorem 4]{Mum} addresses that one can construct isogenies from finite subgroups of an abelian variety and vice versa.

\begin{thm}\cite[II.7 Theorem 4]{Mum}
Let $X$ be an abelian variety. There is a one-to-one correspondence between the two sets of objects:

\begin{itemize}
	\item[(a)] finite subgroups $K \in X$,
	\item[(b)] separable isogenies $f: X \to Y$, where two isogenies $f_{1}: X \to Y_{1}$, $f_{2}: X \to Y_{2}$ are considered equal if there is an isomorphism $h: Y_{1} \to Y_{2}$ such that $f_{2}=h \circ f_{1}$, which is set up by $K=\operatorname{ker}(f)$, and $Y=X/K$.

\end{itemize}

\end{thm}

\vspace{1em}
\noindent{\textbf{The maximal isotropic subgroups.}} In order to count the number of abelian varieties isogenous to $A$, a natural way is to look at all its finite subgroups $G \subset A$. Let $A[m]$ be the $m$-torsion subgroup of $A$ . when $p \nmid m$, $A[m]=(\mathbb{Z}/m\mathbb{Z})^{4}$. Without loss of generality, let $m$ be a prime integer. Recall that for a symplectic $F$-vector space $V$ equipped with the symplectuc bilinear form $\omega: V \times V \to F$, a subspace $H$ is called \textit{isotropic} if for any $h_1, h_2 \in H$, $\omega(h_{1},h_{2})=0$. It is a standard fact that for a symplectic vector space of dimension $2g$, each of the maximal isotropic subspaces is of dimension $g$.

\begin{defn}
Let $(A,\lambda_{A})$ be a principally polarized abelian surface, and $\ell$ be a prime such that $\ell \nmid p$. Define the \textit{$(\ell^{m}, \ell^{m})$-isogeny} to be any isogeny $f: A \to B$ whose kernel is a maximal isotropic subgroup of $A[\ell^{m}]$, with respect to the Weil paring induced by the polarization $\lambda_{A}$. 
\end{defn}

We claim that for an $(\ell^{m}, \ell^{m})$-isogeny $f: A \to B$, there is a unique principal polarization on $B$, denoted by $\lambda_{B}$, such that $f^{*}\lambda_{B}= \ell^{m} \times \lambda_{A}$. This is a consequence of Grothendieck's descent, and we omit the proof here. See \cite[Proposition 2.4.7]{Rob} for detailed proof. This fact allows us to compute a lower bound of $N(q^{n},A)$ first by counting the number of maximal isotropic subgroups of $A$ that are defined over $F_{q^{n}}$, then by computing the number of the subgroups that give the same quotient up to isomorphism.

\begin{lem}\label{counting max iso planes}
For $l \nmid p$, there are $$\ell^{3m}+\ell^{3m-1}+\ell^{3m-2}+l^{3m-3}$$ maximal isotropic subgroups of $A[\ell^{m}]$ with respect to the principal polarization $\lambda_{A}$.
\end{lem}

\begin{proof}
Without loss of generality, one can assume that $\lambda_{A}$ is given by the matrix

$$\lambda_{A}=\left[\begin{array}{cccc}
0 & 1 & 0 & 0 \\
-1 & 0 & 0 & 0 \\
0 & 0 & 0 & 1 \\
0 & 0 & -1 & 0
\end{array}\right]$$
up to a proper choice of basis for the $\ell$-adic Tate module $T_{\ell}A$. Then the corresponding symplectic form is $\psi(x,y)=x^{T}My$. It is easy to see that any cyclic group $H$ of order $\ell^{m}$, we call it \textit{isotropic line}, is an isotropic subgroup. Any maximal isotropic subgroup has the form $(\mathbb{Z}/\ell^{m}\mathbb{Z})^{2}$. These can be viewed as the \textit{isotropic planes} inside $A[\ell^{m}]$. Let $H^{\perp}$ denote the orthogonal complement of $H$. A direct computation shows that 
$$H \subset H^{\perp}, dim(H^{\perp})=3.$$

Since any maximal isotropic subgroup has dimension two, the number of isotropic planes containing $H$ is the number of lines in $H^{\perp}/H$ counts to $\ell^{m}+ \ell^{m-1}$. The number of lines $L$ in $A[\ell^{m}]$ is $\ell^{3m} + \ell^{3m-1} +\ell^{3m-2} +\ell^{3m-3}$ and any maximal isotropic plane contains $\ell^{m}+\ell^{m-1}$ lines. The result follows.
\end{proof}

We introduce a criterion by Waterhouse \cite[Proposition 3.6]{Wat69}, which enables us to rule out maximal isotropic subgroups that give the same quotient variety up to isomorphism. We investigate the $\ell$-power subgroups of $A$, where $\ell \nmid p$. Let $H_{1}, H_{2} \cong (\mathbb{Z}/\ell^{m}\mathbb{Z})^{2}$ be isotropic planes in $A[\ell^{m}]$.

\begin{defn}
	$H_{1}$ is equivalent to $H_{2}$ if they define the same quotient up to isomorphism $$A/H_{1} \cong A/H_{2}.$$

\end{defn}

\begin{thm}\cite[Proposition 3.6]{Wat69}\label{Waterhouse}
Let $G_{1}$ and $G_{2}$ be two finite subgroups of $A$, not necessarily \'etale. Then $A/G_{1} \cong A/G_{2}$ if and only if for some isogeny $\rho \in \End(A)$
 and some non-zero $N \in \mathbb{Z}$, $\rho^{-1}(G_{1})=[N]^{-1}G_{2}$. 
 
 \end{thm}
 
 \begin{proof}
 See \cite[Proposition 3.6]{Wat69}. We include the proof here for completeness.
 
 Suppose $A / G_1 \simeq A / G_2$. Then we have $\varphi_i: A \rightarrow B$ with $\operatorname{ker} \varphi_i=G_i, i=1, 2$. For $N_1$ large (e.g., $N_1=\operatorname{rank} G_1)$, we have $[N_1]^{-1} G_2 \supseteq G_1$. Now $[N_1]^{-1} G_2=\operatorname{ker}(N_1 \varphi_2)$, so by the definition of quotient there is a $\sigma: B \rightarrow B$ such that $\sigma \varphi_1=N_1 \varphi_2$. For $N_2$ large enough there is a $\rho: A \rightarrow A$ with $\varphi_1 \rho= [N_2]\sigma \varphi_1$ (choose an $i_{A}$ and look at the two lattices in E). Thus $\varphi_1 \rho=N_1 N_2 \varphi_2$. Set $N=N_1 N_2$, then
$$
\rho^{-1} (G_1)=\operatorname{ker} (\varphi_1 \rho) =\operatorname{ker} ([N] \varphi_2)=[N]^{-1} G_2 .
$$
Conversely, $A \stackrel{\rho}{\rightarrow} A \rightarrow A / G_1$ shows that $$A / G_1 \simeq A / \rho^{-1} (G_1);$$ likewise $$A / G_2 \simeq A / [N]^{-1} G_2,$$ so the condition is sufficient. 
 	
 \end{proof}
 
 \subsection{Counting inequivalent maximal isotropic planes }
 In this section, we prove a lower bound for the number of inequivalent maximal isotropic planes defined over $\overline{\F}_{q^n}$. The main result is Proposition \ref{type 1} and Proposition \ref{type 2}. For any prime $\ell \nmid p$, fix a basis $\{e_{1},e_{2},f_{1},f_{2}\}$ for $T_{\ell}A$, such that for $i,j=1,2$, $\omega(e_{i},f_{j})=1$ only when $i \ne j$. For the rest of the paper, $H$ will denote a maximal isotropic geometric subgroup of $A[\ell^m]$. 
 
  Let $\phi: A \to B$ be an isogeny defined over $\F_{q^{n}}$. Since there is no isogeny between ordinary and supersingular elliptic curves, the endomorphism ring decomposes as $\End(A)=\End(E) \times \End(E_{ss})$. Therefore there is a decomposition of $\phi$ into ordinary and supersingular part, namely $\phi=\phi_{\operatorname{ord}} \times \phi_{\operatorname{ss}}$, accordingly a decomposition of the kernel: $\ker(\phi)=K_{\operatorname{ord}} \times K_{ss}$, where $K_{\operatorname{ord}} \subset E$ and $K_{ss} \subset E_{ss}$. We have the following theorems on the number of endomorphisms of elliptic curves over finite fields whose kernel is cyclic: 
  
\begin{prop}\label{ordinary}
Let $E$ be an ordinary elliptic curve defined over $F_{q^{n}}$. For any positive integer $d$, the number of endomorphisms in $\End(E)$ with cyclic kernel $\mathbb{Z}/d\Z$ is bounded up by $O(d^{\epsilon})$.	
\end{prop}

\begin{proof}
	Let $K=\End^{\circ}(E)$, $E$ is ordinary implies that $K$ is a quadratic number field. Let $\mathcal{O}_{K}$ denote the ring of integers of $K$.
 Assume that $d$ has prime decomposition $d=p_{1}^{e_1} \cdots p_{r}^{e_r}q_{1}^{f_1} \cdots q_{s}^{f_s}d_{1}$, such that $p_{i}$ splits in $\mathcal{O}_{K}$, $q_{j}$ inert in $\mathcal{O}_{K}$ and every prime factor of $d_{i}$, namely the ramified prime, divides $D$.    
The number of endomorphisms in $\End(E)$ with cyclic kernel $\mathbb{Z}/\ell^{m}\mathbb{Z}$ is the number of elements in $\mathcal{O}_K$ with norm $\ell^{k}$. Therefore it is $(e_{1}+1) \cdots (e_{r}+1)$ if all $f_{j}$, $1 \le j \le s$ are even, or zero otherwise. The divisor bound is $O(d^{\epsilon})$, which is a standard fact.
\end{proof}

\begin{prop}\label{SScyclic}
Let $E_{ss}$ be a supersingular elliptic curve defined over $F_{q^{n}}$. For $\ell \nmid D$ where $D$ is the determinant of the norm form on $End(E_{ss})$, there are $O(\ell^{m})$ endomorphisms whose kernel is the cyclic group $\mathbb{Z}/\ell^{m}\mathbb{Z}$. 
\end{prop}

\begin{proof}
			
	Let $E_{ss}$ be a supersingular elliptic curve defined over $F_{q^{n}}$ with characteristic $p$, then $O_{E_{ss}}=\End(E_{ss})$ is a maximal order in the quaternion algebra ramified exactly at $p$ and $\infty$. Endomorphism with kernel a cyclic subgroup of order $m$, i.e., of degree $m$, are elements in $O_{E_{ss}}$ with norm $m$. 
		For a quaternion algebra $F=\mathbb{Q}+\mathbb{Q}\alpha +\mathbb{Q}\beta +\mathbb{Q}\alpha\beta$ where $\alpha^{2}=a,\beta^{2}=b, a<0, b<0, \beta\alpha=-\alpha\beta$, $x=x_{0}+x_{1}\alpha+x_{2}\beta+x_{3}\alpha\beta \in F$, the norm $N(x)=x\bar{x}=x_{0}^{2}-ax_{1}^{2}-bx_{2}^2+abx_{3}^2$ is a quaternion quadratic form. The question boils down to counting the number of representations $$r(n)=r_{N}(n)=\#\{x\in \mathbb{Z}^{4},x=(x_{0},x_{1},x_{2},x_{3});N(x)=n\}$$
		This can be solved making use of the theta series $$ \vartheta(z)=\sum_{\alpha \in \mathbb{Z}^{4}}e(zN(\alpha))=\sum_{n \ge 0}r(n)e(nz) $$
		where $e(z)=e^{2\pi iz} $, which is a generating series for $r(n)$. $\vartheta(z)$ satisfy 
		$$\vartheta(\frac{az+b}{cz+d})=\chi(\gamma)(cz+d)^{m/2}\vartheta(z) $$
		where $\gamma \in SL_{2}(\mathbb{Z})$ and therefore is a modular form of weight $m/2$. So it can be written as the sum of an Eisenstein series 
		
		$$E(z)=\sum_{n \ge 0}\rho(n)e(nz),\rho(0)=1$$
		
		 and a cusp form 
		 $$f(z)=\vartheta(z)-E(z)=\sum_{n \ge 1}a(n)e(nz).$$
		 
		 Thus we can write $r(n)=\rho(n)+a(n)$ and then bound it from above by estimating the Fourier coefficient $a(n)$ of the cusp form and estimating $\rho(n)$ gives a bound from below. 
		 In our case where $m=4$, assume that $\ell \nmid D$. We have $\frac{m}{2}-1=1$ such that
		 $$\rho(n) \gg n.$$ 
		 One way to get a nontrivial upper bound for $a(n)$ is to use the Rankin-Selberg method. For even $m$, Deligne[Del73] proved that $a(n) \ll n^{\frac{m}{4}- \frac{1}{2}+\epsilon}$. In our case, it turns out to be 
		 $$a(n) \ll n^{\frac{1}{2}}.$$
		 
		 Putting together, we get 
		 $$r(\ell^{m}) \gg \ell^{m}.$$
		 
		 Since $E_{ss}$ is simple, every non-zero endomorphism is an isogeny, and we have $$\ell^{m}+\ell^{m-1} =O(\ell^{m})$$
		  cyclic subgroups of order $\ell^{m}$ in $E_{ss}[\ell^{m}]=\mathbb{Z}/\ell^{m}\mathbb{Z} \times \mathbb{Z}/\ell^{m}\mathbb{Z}$. Thus 
		  $$r(n)=r_{N}(n)=O(\ell^{m}).$$
		 	 
\end{proof}

There are two types of maximal isotropic planes in $A[\ell^{m}]$ we take into concern with respect to our choice of basis:

\begin{itemize}
	\item \textbf{Type 1:} $H$ is a product $H_{1} \times H_{2}$ where $H_{1} \in E$, $H_{2} \in E_{ss}$.

	\item \textbf{Type 2:} $H$ cannot be written as a product $H_{1} \times H_{2}$ where $H_{1} \le E$, $H_{2} \le E_{ss}$.
\end{itemize}
 \subsubsection{$H$ is of product type }
 In the case where $H$ is of type $1$, we write $H=\left \langle ae_{1}+be_{2},cf_{1}+df_{2} \right \rangle$, where $a,b,c,d\in \mathbb{Z}/\ell^{m}\mathbb{Z}$. Here $\{e_{1}, e_{2}, f_{1}, f_{2}\}$ denote a basis of  $A[\ell^m]$. We claim that:
 
 \begin{prop}\label{type 1}
 	Let $N_{1}$ be the number of inequivalent maximal isotropic planes of type $1$. We have $$N_{1} \asymp \ell^{m}$$ 
 \end{prop}
\begin{proof}

 For an elliptic curve, either ordinary or supersingular, there are $O(\ell^{m})$ cyclic subgroups of the order less than or equal to $\ell^{m}$. Therefore, there are $O(\ell^{2m})$ such kind of $H$ in total. Let $H_{1}$ and $H_{1}^{'}$ be cyclic subgroups(isotropic lines) of $E[\ell^{m}]$. By Theorem \ref{Waterhouse}, $A/H_{1} \cong A/H_{1}^{'}$ if and only if there exists $\phi \in \End(E)$, $N \in \mathbb{Z}$ such that $\phi^{-1}H_{1}^{'}=[N]^{-1}H_{1}$.
  For such possible $\phi$ that has prime-to-$\ell$ kernel, we have $N \nmid \ell$, $N^{-1}H_{1}=(\mathbb{Z}/N\mathbb{Z}) \times (\mathbb{Z}/N\mathbb{Z}) \times (\mathbb{Z}/\ell^{m}\mathbb{Z})$. Since $\Ker(N) \subset \Ker(\phi)$, $\phi$ factors through the multiplication by $N$ map as $\phi=i\circ N$ where $i \in \Aut(E)$. But for an ordinary elliptic curve, there are only finitely many units in $\End(E)$, thus the possible choices of $\phi$.

   The same argument also works for $\phi$ has $\ell$-power kernel. Indeed, for positive integer $k$,  we have $[\ell^{k}]^{-1}H_{1}=(\mathbb{Z}/\ell^{k}\mathbb{Z}) \times (\mathbb{Z}/\ell^{k+m}\mathbb{Z})$ and the possible choices of $\Ker(\phi)$ are $(\mathbb{Z}/\ell^{k+i}\mathbb{Z}) \times (\mathbb{Z}/\ell^{k-i}\mathbb{Z})$ for $0 \le i \le m$. Proposition \ref{ordinary} implies that the number of inequivalent isotropic lines $H_{1} \subset E$ is $O(l^{m-\epsilon})$.	 
	 
 We assumed that $H_{2}$ comes from the supersingular elliptic curve $E_{ss}$. Since the number of supersingular elliptic curves up to $\overline{\mathbb{F}}_{q}$-isomorphism is finite, for instance, see \cite[V.4 Theorem 4.1]{Sil09}. we have finitely many inequivalent $H_{2}\in E_{ss}$. Putting these arguments together, we proved that the number of such inequivalent $H$ of type 1 is asymptotically $\ell^{m}$.
 
\end{proof}

 \subsubsection{$H$ is not a product}
 
In the second case, we assume that $H$ is not a product of ordinary and supersingular subgroups. To be explicit, we write $H=\langle e_{1}+af_{1}+bf_{2}, e_{2}+cf_{1}+df_{2} \rangle$, with the assumption that $$det(\begin{bmatrix}
 	a & b \\
 	c & d 
 	
 \end{bmatrix})=-1.$$
 By Lemma \ref{counting max iso planes} and Proposition \ref{type 1}, the total number of the maximal isotropic plane in the non-product form is $O(\ell^{3m})$.

Fix an $H$ in this form. We count the number of all isotropic planes $H^{\prime}$ that is equivalent to $H$. By work of Waterhouse \cite{Wat69}, $$\phi^{-1}H^{\prime}=[N]^{-1}H$$ for some $\phi \in \End(A)$ and some positive integer $N$. Before proving Proposition \ref{type 2}, we introduce the following lemma.

 For any $\phi \in \End(A)$, we can write $\phi=\phi_{\operatorname{ord}} \times \phi_{\text{ss}}$. As a consequence, the kernel of $\phi$ decomposes as $\Ker(\phi)=K_{\operatorname{ord}} \times K_{\operatorname{ss}}$. Therefore, to bound the number of endomorphisms $\phi$ once we fix $N$, we need to bound the number of possible $\phi_{\text{ord}}$ and $\phi_{\text{ss}}$ separately. By Proposition \ref{ordinary}, the number of endomorphisms of an ordinary elliptic curve with a fixed degree $d$ is $O\left(d^\epsilon\right)$. Therefore, we only have to determine how many possible choices of $K_{\operatorname{ss}}$ we can have under the assumption of $H$.

\begin{lem}\label{no supersingular order greater than m}
	Assume that $H$ is of Type $2$, and take $N= \ell^{m}$. Then there are at most $O(\ell^m)$ supersingular endomorphisms which we denote by $\phi_{\text{ss}}$, such that $$\phi_{\text{ss}}=\ell^{a} \circ \phi_{\text{cyc}},$$ for some $0 \le a \le m$, and there exists an endomorphism $\phi=\phi_{\operatorname{ord}} \times \phi_{\text{ss}}$, such that $$\phi^{-1}H^{\prime}=[N]^{-1}H.$$ Moreover, $\phi_{\text{cyc}}$ is cyclic of order at most $\ell^{m}$.
	 
\end{lem}

\begin{proof} 
First of all, we prove that the degree of $\phi_{\text{cyc}}$ is at most $\ell^{m}$.
	This is equivalent to the statement that we cannot have an element $$\alpha \in [\ell^{m}]^{-1}H \cap E_{\text{ss}}$$ whose order is greater than or equals to $\ell^{m+1}$. We prove this by contradiction. Suppose such an element exists and call it $|x|$. Since $|x| \ge \ell^{m+1}$, $\ell^{m}\circ (x)$ is nontrivial. By definition we have $\ell^{m} \circ (x) \in H $ and $\ell^{m} \circ (x) \in E_{\text{ss}}$. Therefore $$\ell^{m}\circ (x) \in [\ell^{m}]^{-1}H \cap E_{\text{ss}}.$$
	Since $H$ has the form $H=\langle e_{1}+af_{1}+bf_{2}, e_{2}+cf_{1}+df_{2} \rangle$, one gets to the conclusion that $E_{\text{ss}} \cap H =\{\text{id}\}$. Hence the contradiction.
	
	By Proposition \ref{SScyclic}, we conclude that there are at most $O(\ell^{m})$ such $\phi_{\text{cyc}}$, hence at most $O(\ell^{m})$ such $\phi_{\text{ss}}$.

		\end{proof}

\begin{prop}\label{type 2}
	Let $N_2$ be the number of inequivalent maximal isotropic planes of type 2. We have $$N_2 >>\ell^{2m-\epsilon}$$
\end{prop}

\begin{proof}
 	
For a fixed $H$, getting a lower bound of the number of inequivalent isotropic planes is equivalent to getting an upper bound of the maximal isotropic planes which are equivalent to $H$. We do this by bounding the number of endomorphisms $\phi \in \End(A)$ such that $\phi \circ [N]^{-1}H$ is a maximal isotropic plane for each fixed $N$, as $N$ goes through the set of positive integers.

First, suppose that $\ell \nmid N$. In this case, the pullback of an isotropic plane under $\phi$ has the form $$\phi^{-1}(H) \simeq (\Z/\ell^{m}\Z)^{2} \times \ker(\phi).$$
On the other hand, we have $$[N]^{-1}H \simeq (\Z/\ell^{m}\Z)^{2} \times (\Z/N\Z)^{4}.$$

By Theorem \ref{Waterhouse}, $\Ker(\phi) \simeq (\Z/N\Z)^{4}$. Therefore we have $\phi=i\circ N$, where $i \in \Aut(A)$ is an automorphism. Taking for granted the fact that principally polarized abelian varieties have finitely many automorphisms which are independent of $n$, we get finitely many $H^{\prime}$ that is equivalent to $H$ where $H^{\prime}=\phi (\circ [N]^{-1} H)$ is a maximal isotropic plane inside $A[\ell^m]$.

When $l \mid N$, $k \ge 1$. We may write $N=N_{0} \cdot \ell^{a}$ for some $a \ge 1$, where $N_0$ is coprime to $\ell$. Then $$[N]^{-1}H \simeq (\Z/\ell^{a}\Z)^{2} \times (\Z/\ell^{m+a}\Z)^{2} \times (\Z/N_{0}\Z)^{4}.$$ Therefore $$\Ker(\phi) \simeq (\Z/N_{0}\Z)^{4} \times G_{\ell},$$ where $G_{\ell}$ is some $\ell$-power subgroup which we will specify later. Therefore $\phi$ can be written as a decomposition $$\phi=i \circ N_{0} \circ \phi_{l}$$ where $\phi_{\ell}$ is an $\ell$-power isogeny with kernel $G_{\ell}$.

So without loss of generality, we can assume that $N=\ell^{k}$ for some $k \ge 1$ and prove the following subcases depending on the power of $\ell$. 

If $k<m$, we have 
$$[\ell^{k}]^{-1}H=(\mathbb{Z}/\ell^{k}\mathbb{Z})^{2} \times (\mathbb{Z}/\ell^{k+m}\mathbb{Z})^{2}.$$ 
Let $A \subset [\ell^{k}]^{-1}H$ be a subgroup such that $$[\ell^{k}]^{-1}H/A \simeq (\Z/\ell^{m}\Z)^{2}.$$
Then $$A \simeq (\mathbb{Z}/\ell^{k}\mathbb{Z})^{2} \times (\mathbb{Z}/\ell^{k+i}\mathbb{Z}) \times (\mathbb{Z}/\ell^{k-i}\mathbb{Z})$$ for some $0 \le i \le m$. Hence the possible choices of $\Ker(\phi)$ have the above form.

For $k=m$, we have $$[\ell^{m}]^{-1}H=(\mathbb{Z}/\ell^{m}\mathbb{Z})^{2} \times (\mathbb{Z}/\ell^{m+m}\mathbb{Z})^{2}.$$ Similarly we have the possible choices for $\Ker(\phi)$ are subgroups in the following form
$$(\mathbb{Z}/\ell^{m+i}\mathbb{Z}) \times (\mathbb{Z}/\ell^{m-i}\mathbb{Z}) \times (\mathbb{Z}/\ell^{m+j}\mathbb{Z}) \times (\mathbb{Z}/\ell^{m-j}\mathbb{Z})$$
for $0 \le i \le m$ and $0 \le j \le m$.

For $k > m$, the possible choices for $\Ker(\phi)$ are
$$(\mathbb{Z}/\ell^{k+i}\mathbb{Z}) \times (\mathbb{Z}/\ell^{k+j}\mathbb{Z}) \times (\mathbb{Z}/\ell^{k+m-n}\mathbb{Z}) \times (\mathbb{Z}/\ell^{k+m-w}\mathbb{Z})$$
where $i,j,n,w \ge 0$ and $i+j+n+w=2m.$
    
However, since $\Z/\ell^{k-m}\Z$ is a common factor, this implies $\Ker(\phi)$ contains $\ker([\ell^{k-m}])=(\mathbb{Z}/\ell^{k-m}\mathbb{Z})^{4}$. Therefore $\phi$ factors through the multiplication by $\ell^{k-m}$ map, we are returning to the case where $k=m$.

\vspace{1em}

Now we can bound the number of endomorphisms and the number of maximal isotropic planes equivalent to a given $H$. Proposition \ref{ordinary} asserts that the number of endomorphisms of an ordinary elliptic curve with a fixed degree $d$ is $O\left(d^\epsilon\right)$ and Lemma \ref{no supersingular order greater than m} states that there are at most $O(\ell^m)$ supersingular endomorphisms that serve as the supersingular part of $\phi$. An upper bound of the maximal isotropic planes isomorphic to a fixed $H$ is $O(\ell^{m + \epsilon})$. Therefore the total number of inequivalent maximal isotropic planes of type $2$ in $A[\ell^{m}]$ is 
$$O(\ell^{2m-\epsilon})=O(\ell^{3m})/O(\ell^{m+ \epsilon})).$$

\end{proof}

\begin{remark}
	We note that improving the bound without the $\epsilon$ term is plausible.
\end{remark}

\vspace{1em}
\begin{cor}\label{product}
	Let $\ell_{1}, \cdots, \ell_{n}$ be $n$ primes different from $p$ and let $m_{1}, \cdots, m_{n}$ be positive integers. Let $N_0$ be the total number of inequivalent maximal isotropic planes. We have
	$$N_{0} >> (\ell_{1}^{m_{1}}\cdots\ell_{n}^{m_{n}})^{2-\epsilon}$$  
\end{cor}

\begin{proof}
	The proof is a generalization of Proposition \ref{type 2} proof. Since the majority of the inequivalent maximal isotropic planes come from products of isotropic planes of Type $2$ as $\ell$ varies, we fix a subgroup $G$  
	$$G \simeq (\Z/\ell_{1}^{m_{1}}\Z)^{2} \times \cdots \times (\Z/\ell_{n}^{m_{n}}\Z)^{2}$$
	 such that for each $\ell_{i}$, $1 \le i \le n$, $(\Z/\ell_{1}^{m_{1}})^{2}$ is an isotropic plane of Type $2$. Let $\{e^{1}_{1}, e^{1}_{2}, f^{1}_{1}, f^{1}_{2}\}, \cdots , \{e^{n}_{1}, e^{n}_{2},  f^{n}_{1}, f^{n}_{2}\}$ be a basis for $T_{\ell_{1}}(A), \cdots, T_{\ell_{n}}(A)$, respectively.
	
	 We count the maximal number of maximal isotropic planes $G^{\prime}$ that is equivalent to $G$. By Theorem \ref{Waterhouse}, $G$ and $G^{\prime}$ are equivalent if there is $\phi \in \End(A)$ and non-zero positive integer $N$ such that $\phi^{-1}G^{\prime}=[N]^{-1}G$. We split the argument into different cases based on the choice of $N$.
	 
	 \vspace{1em}
	\noindent \textbf{Case I: $N$ is coprime to $\ell_{1}, \cdots, \ell_{n}$.}
	
 If $N$ is coprime to $\ell_{1} \cdots \ell_{n}$,
 $$[N]^{-1}H=(\mathbb{Z}/N\mathbb{Z})^{4} \times (\Z/\ell_{1}^{m_{1}})^{2}\Z \times \cdots \times (\Z/\ell_{n}^{m_{n}}\Z)^{2}.$$ Therefore we have $\phi=i\circ N$, where $i \in \Aut(A)$ is an automorphism. Taking for granted the fact that principally polarized abelian varieties have finitely many automorphisms, we get finitely many $H^{\prime}$ that is equivalent to $H$ under the assumption that $H^{\prime}=\phi \circ [N]^{-1} H$.

	\vspace{1em}
	\noindent \textbf{Case II: $N= \ell_{j_{1}}^{m_{j_{1}}}\cdots \ell_{j_{k}}^{m_{j_{k}}}$ for some $ 0< k \le n$ and $1 \le j_{1} < \cdots < j_{k} \le n$.}

Similar to Proposition \ref{type 2}, if $N$ is not coprime to some of the $\{\ell_{1}, \cdots, \ell_{n}\}$, we may restrain ourselves on this case, for the same reason as explained in the proof of Proposition \ref{type 2}.

 The pullback  of $G$ under $\ell_{j_{1}}^{m_{j_{1}}}\cdots \ell_{j_{k}}^{m_{j_{k}}}$ is isomorphic to 
 $$(\Z/\ell_{j_{1}}^{m_{j_{1}}}\Z)^{2} \times \cdots \times (\Z/\ell_{j_{k}}^{m_{j_{k}}}\Z)^{2} \times (\Z/\ell_{j_{1}}^{2m_{j_{1}}}\Z)^{2} \times \cdots \times (\Z/\ell_{j_{k}}^{2m_{j_{k}}}\Z)^{2} \times \prod_{j \ne j_{1}, \cdots, j_{k}}(\Z/\ell_{j}^{m_{j}}\Z)^{2}.$$ 
 Recall that for each $1 \le i \le n$ we assume that $$G_{i}:= (\Z/\ell_{i}^{m_{i}}\Z)^{2}= \langle e^{i}_{1}+a_{i}f^{i}_{1}+b_{i}f^{i}_{2}, e^{i}_{2}+c_{i}f^{i}_{1}+d_{i}f^{i}_{2} \rangle,$$ with the assumption that $$\det(\begin{bmatrix}
 	a_{i} & b_{i} \\
 	c_{i} & d_{i} 
 	
 \end{bmatrix})=-1.$$
 Similar to the proof of Lemma \ref{no supersingular order greater than m}, an endomorphism $\phi$ that satisfies the Waterhouse's criterion can be realized as an endomorphism with kernel $\Ker(\phi)=K_{\operatorname{ord}} \times K_{\operatorname{ss}}$. Moreover, we can factor out the $\ell$-power scalar multiple from each part and consider those supersingular endomorphisms whose kernels are cyclic subgroups. We claim that the supersingular part $K_{\operatorname{ss}}$ contains a cyclic subgroup of order at most $\ell_{j_{1}}^{m_{j_{1}}}\cdots \ell_{j_{k}}^{m_{j_{k}}}$. Suppose this is not the case, i.e., there is an element $x \in [\ell_{j_{1}}^{m_{j_{1}}}\cdots \ell_{j_{k}}^{m_{j_{k}}}]^{-1}H \cap E_{\text{ss}}$ with order $|x| \ge \ell_{j_{1}}^{m_{j_{1}}}\cdots \ell_{j_{k}}^{m_{j_{k}}}$, then $$\ell_{j_{1}}^{m_{j_{1}}}\cdots \ell_{j_{k}}^{m_{j_{k}}} \circ (x) \in \prod_{i} G_i \cap E_{\text{ss}} $$ is nontrivial. By definition of $G_{i}$ for each $1 \le i \le n$, the intersection $\prod_{i} G_i \cap E_{\text{ss}}$ is trivial. Therefore the claim follows.

By Proposition \ref{SScyclic}, there are at most $\ell_{j_{1}}^{m_{j_{1}}}\cdots \ell_{j_{k}}^{m_{j_{k}}}$ such $K_{\operatorname{ss}}$. For the ordinary component, Proposition \ref{ordinary} implies that there are at most $(\ell_{j_{1}}^{m_{j_{1}}}\cdots \ell_{j_{k}}^{m_{j_{k}}})^{2\epsilon}$ many possible choices for $K_{\operatorname{ord}}$. We conclude that there are at most $(\ell_{j_{1}}^{m_{j_{1}}}\cdots \ell_{j_{k}}^{m_{j_{k}}})^{2- \epsilon}$ maximal isotropic planes $G^{\prime}$ that are equivalent to a given $G$. The result follows.
\end{proof}

\section{Proof of Theorems}
\subsection{Semisimplicity assumption on the Frobenius action}
Let $E$ be an ordinary elliptic curve over $\F_{q}$ with $\End^{\circ}(E)=K$ and let $\pi$ be the Frobenius endomorphism. Here $K$ is the quadratic imaginary field generated by $\pi$: $K=\Q(\pi)$. We fix a basis $\{ e_1, e_2 \}$ of $T_{\ell}(E)$. The characteristic polynomial $\chi_{\pi}$ is the unique polynomial such that for every $n$ prime to $p$, the characteristic polynomial of the action of the Frobenius $\pi$ on $E[n]$ is $\chi_{\pi}$ mod $n$. Let $\Delta_{\pi^{n}}$ be the discriminant of $\pi^{n}$. The characteristic polynomial is a quadratic polynomial $$\chi_{\pi^{n}}=x^{2}-t_{n}x+q^{n}.$$ We have $\Delta_{\pi^{n}}= t_{n}^{2}-4q^{n}.$  

\vspace{1em}
\begin{remark}
	For an degree $n$ extension $F_{q^{n}}$, the Frobenius of $E_{\F_{q^{n}}}$ is $\pi^{n}$.
\end{remark}

\vspace{1em}
An isogeny $\phi \colon E \to E^{\prime}$  whose kernel is a cyclic subgroup of $E$ can be understood by looking at the Frobenius action on the torsion subgroups. If $\phi$ is defined over $\F_{q^{n}}$, then $\ker \phi$ is stabilized by the Frobenius action. For $\ell \ne p$, the number of $\ell$-power isogenies defined over $\F_{q^{n}}$ is determined by the action of $\pi^{n}$ on $\ell$-power torsions. Moreover, the action of Frobenius can be realized as a $2 \times 2$ matrix with coefficients in $\Z/\ell^{m}\Z$. 

\vspace{1em}
Now we state the semisimplicity assumption on the Frobenius action, which helps us narrow down cases that we should focus on. 

Recall that our goal is to compute the number of $\ell$-power isogenies from $E$ that is defined over $\F_{q^n}$, as $\ell$ goes through all prime integers. The semisimplicity of the Frobenius action depends on whether $\ell$ is ramified in $\mathcal{O}_{K}$ or not: 
\begin{itemize}
\item[$\star$]If $\ell$ is unramified in $\mathcal{O}_{K}$, then $\pi^{n}$ is semisimple modulo $\ell^{m}$ for all $m \ge 1$. We prove the following lemma:
\end{itemize}
\begin{lem}\label{2m}
Let $m$ be the maximal number such that $\Delta_{\pi^{n}} \equiv 0$ mod $\ell^{2m}$, then 
$$\pi^{n} \equiv \begin{bmatrix}
 	\lambda & 0 \\
 	0 & \lambda 
 	
 \end{bmatrix} \text{  mod  } \ell^{m} .$$	
\end{lem}

\begin{proof}
Let $\lambda_{1}$ and $\lambda_{2}$ be the eigenvalues of $\chi_{\pi^{n}}$. We have $$\Delta_{\pi^{n}}=(\lambda_1 -\lambda_2)^2$$ and $\ell^{2m}$ divides $\Delta_{\pi^{n}}$. Therefore $\ell^{m} \mid (\lambda_1 - \lambda_2).$  
	
	Since $\ell$ is unramified in $\mathcal{O}_{K}$, the action of $\pi^{n}$ is semisimple modulo $\ell^m$. Work on the setting over $\Z_{\ell}$, if $\lambda_{1}, \lambda_{2} \in \Z_{\ell}$ we are done. Otherwise, $\lambda_{1}, \lambda_{2} \in \mathcal{O}_{\ell}$ where $\mathcal{O}_{\ell}$ is unramified of degree $2$. We now prove that $\lambda_{1}, \lambda_{2}$ mod $\ell^{m}$ are in $\Z/\ell^{m}\Z$. By the semisimplicity assumption, the action of $\pi^n$ is diagonalizable over $\mathcal{O}_{\ell}/\ell^{m}\mathcal{O}_{\ell}$ for any $m \ge 1$. This is equivalent to say there exists $X \in \operatorname{GL}_{2}(\mathcal{O}_{\ell}/\ell^{m}\mathcal{O}_{\ell})$ such that 
	$$\pi^{n} = X\begin{bmatrix}
 	\lambda_1 & 0 \\
 	0 & \lambda_2 
 	
 \end{bmatrix}X^{-1} \text{  mod  } \ell^{m} .$$
 But we proved that $\lambda_1 \equiv \lambda_2$ mod $\ell^m.$ Therefore 
 $$\pi^{n} = XX^{-1}\begin{bmatrix}
 	\lambda_1 & 0 \\
 	0 & \lambda_1 
 	
 \end{bmatrix} = \begin{bmatrix}
 	\lambda_1 & 0 \\
 	0 & \lambda_1 
 	
 \end{bmatrix} \text{  mod  } \ell^{m}.$$
 Since $\pi^n \text{ mod } \ell^m \in \operatorname{GL}_{2}(\Z/\ell^{m}\Z),$ the lemma follows. 
\end{proof}

\begin{itemize}
\item[$\star$]If $\ell$ is ramified in $\mathcal{O}_{K}$, then it is possible that the Frobenius action $\pi^n$ is not semisimple modulo $\ell^m$. But recall that we made the assumption that $n$ is an integer coprime to the discriminant $\Delta_{K}$ of $\mathcal{O}_K$. This implies that for all prime $\ell \mid \Delta_{K}$, the power of $\ell$ dividing $\Delta_{\pi^n}$ is bounded independent of $n$. Therefore we have the following corollary:
  
\end{itemize}

\begin{cor}\label{does not grow with n}
	Let $S$ be the set of all primes that divides $\Delta_{K}$. Let $n$ be an integer such that for all $\ell \in S$, $(n, \ell) = 1$. Then the number of $\ell$-power isogenies where $\ell \in S$ is bounded independent of $n$. In other words, this number does not grow with $n$.  
\end{cor}

\vspace{1em}

We make the table of classification of the Frobenius action under the semisimplicity assumption as follows:
\begin{itemize}
	\item[(a).] Assume $\ell, m, n$ such that $\chi_{\pi^{n}}$ mod $\ell^{m}$ is  irreducible modulo $\ell^{m}$. In this case $\pi^{n}$ acts on $E[\ell^{m}]$ as a distortion map. I.e., no subgroups $\Z/\ell^{m}\Z$ is stabilized by $\pi^{n}$. Therefore $E$ has no $\ell^{m}$-isogenies defined over $\F_{q^{n}}.$  
	\item[(b.1).] Assume $\ell, m, n$ such that $\pi^{n}$ is diagonalizable mod $\ell^{m}$. Moreover, $\chi_{\pi^{n}}$ has distinct eigenvalues $\lambda$ and $\mu$ modulo $\ell^{m}$. In this case, the Frobenius acts on $E[\ell^{m}]$ as a matrix conjugates to $\begin{bmatrix}
 	\lambda & 0 \\
 	0 & \mu 
 	
 \end{bmatrix}$ and there are two isogenies of degree $\ell^{m}$ from $E$ which are defined over $\F_{q^n}$, given by $E$ modulo cyclic subgroups generated by the two eigenvectors of $\lambda$ and $\mu$ respectively. 
 
	\item[(b.2).] Assume $\ell, m, n$ such that $\pi^{n}$ is diagonalizable mod $\ell^{m}$. Moreover, the Frobenius $\chi_{\pi^{n}}$ modulo $\ell^{m}$ has one eigenvalue $\lambda$ of multiplicity two. In this case, the Frobenius acts on $E[\ell^{m}]$ as a scalar multiple by $\begin{bmatrix}
 	\lambda & 0 \\
 	0 & \lambda 
 	
 \end{bmatrix}$ and every $\ell^{m}$ subgroup is stable under $\pi^{n}$. Therefore, there are $\ell^{m}+ \ell^{m-1} +\cdots +1$ $\ell$-power isogenies of degree less than or equal to $\ell^{m}$ from $E$ which are defined over $\F_{q^n}$.
 
 \item[(b.3).] Assume $\ell, m, n$ such that $\pi^{n}$ is diagonalizable mod $\ell^{m}$. Assume that $\chi_{\pi^{n}}$ has distinct eigenvalues $\lambda$ and $\mu$ modulo $\ell^{m}$ but eigenvalues are congruent modulo $\ell^r$ for some $1<r<m$. In this case, the Frobenius acts on $E[\ell^{r}]$ as a matrix conjugates to $\begin{bmatrix}
 	\lambda & 0 \\
 	0 & \lambda 
 	
 \end{bmatrix}$ and there are $\ell^{r}+ \ell^{r-1} +\cdots +1$ $\ell$-power isogenies of degree less than or equal to $\ell^{m}$ from $E$ which are defined over $\F_{q^n}$.

	\end{itemize}
 
\subsection{Horizontal isogenies}

As one may notice, there are a lot of prime power isogenies of $E$ that are indeed endomorphisms of $E$.  By the theory of complex multiplications, the number of such isogenies is bounded by the class number of $\mathcal{O}_{K}$, see Theorem \ref{classBound}. We give information about when such isogenies arise and use the information to bound the total number of isogenies defined over $\F_{q^n}.$ 

We use the classification of the Frobenius action on the $\ell$-torsion subgroups to compute each case's horizontal prime power isogenies.
\begin{defn}
	Let $f \colon E \to E^{\prime}$ be an isogeny of degree $\ell^{m}$. We say $f$ is \textit{horizontal} if $\End(E)=\End(E^{\prime}).$  
\end{defn}

Let $\mathfrak{a}$ be an invertible ideal in $\End(E)$. Define the $\mathfrak{a}$-torsion subgroup of $E$ as  
$$E[\mathfrak{a}]:=\left\{P \in E\left(\overline{\mathbb{F}}_q\right) \mid \sigma(P)=0 \text { for all } \sigma \in \mathfrak{a}\right\}.$$
Let $\phi_{\mathfrak{a}}$ be an isogeny whose kernel is $E[\mathfrak{a}]$. Then the codomain $E/E[\mathfrak{a}]$ is a well-defined elliptic curve. The isogeny $\phi_{\mathfrak{a}}$ is horizontal, and its degree equals the ideal norm of $\mathfrak{a}$. We denote by $\mathfrak{a} \cdot E$ for the isomorphism class of the image of $\phi_{\mathfrak{a}}.$ 

\begin{lem}\label{horizontal}
	Let $E_{q^{n}}$ be an ordinary elliptic curve over $\F_{q^{n}}$ with the Frobenius action by $\pi^{n}.$  Let $H(\ell^{m})$ denote the number of horizontal $\ell^{m}$-isogenies.
$$H(\ell^{m})=
\begin{cases}
0, \text{  if  $\pi^{n}$ is irreducible }\\
1, \text{  if  $\pi^{n}$ is diagonalizable with one eigenvalue modulo $\ell^{m}$ }\\
2, \text{  if  $\pi^{n}$ is diagonalizable with two eigenvalues modulo $\ell^{m}$ }.
\end{cases}
$$
\end{lem}

\begin{proof}
	If $\Delta_{\pi^{n}}$ is not a square modulo $\ell^{m}$, we are in case (a) where no subgroup of order $\ell^{m}$ is stabilized by the action of $\pi^{n}$. Therefore no $\ell^{m}$-isogeny is defined over $\F_{q^{n}}$.
	
	Suppose $\pi^{n}$ is diagonalizable with one eigenvalue modulo $\ell^{m}$. In that case, we are in case (b.2) (and (b.3)), and there is one horizontal isogeny given by $\mathfrak{a}=(\pi^{n}-\lambda, \ell^{m})$ with norm $\ell^{m}$. Moreover, $\phi_{\mathfrak{a}}$ is self-dual.
	
	If $\pi^{n}$ is diagonalizable with two eigenvalues modulo $\ell^{m}$, we are in case (b.1). There are two torsion subgroups of order $\ell^{m}$, generated by the eigenvector of $\lambda$ and $\mu$, respectively. The two horizontal isogenies are given by the ideals $\mathfrak{a}=(\pi^{n}-\lambda, \ell^{m})$ and $\hat{\mathfrak{a}}=(\pi^{n}-\mu, \ell^{m})$. Furthermore, $\mathfrak{a}\hat{\mathfrak{a}}=(\ell^{m})$ implying that $\mathfrak{a}$ and $\hat{\mathfrak{a}}$ are the inverse of one another in the class group, thus $\phi_{\hat{\mathfrak{a}}}$ is the dual isogeny of $\phi_{\mathfrak{a}}$.  

\end{proof}

Recall that for an elliptic curve $E$ with CM by an order $\mathcal{O}$, horizontal $\ell$-isogenies correspond to the CM action of an invertible $\mathcal{O}$-ideal of norm $\ell$. Moreover, let $\operatorname{Ell}_{q}(\mathcal{O})$ be the set 
$$
\operatorname{Ell}_{q}\left(\mathcal{O}\right):=\left\{E / \mathbb{F}_q: \operatorname{End}(E) \simeq \mathcal{O}\right\}.
$$
Because elliptic curves in $\operatorname{Ell}_{q}(\mathcal{O})$ are connected exclusively by horizontal cyclic isogenies, the theory of complex multiplication tells us:
\begin{thm}\label{classBound}
Let $E$ be an elliptic curve with endomorphism ring $\mathcal{O}$. Then the set of horizontal isogenies forms a principal homogeneous space under the class group of $\mathcal{O}$. To be precise, assume the set is non-empty. Then it is a principal homogeneous space for the class group $\mathcal{C}\ell(\mathcal{O})$, under the action
\begin{align}
\mathcal{C}\ell(\mathcal{O}) \times \operatorname{Ell}_q(\mathcal{O}) & \longrightarrow \operatorname{Ell}_q(\mathcal{O}), \\
(\mathfrak{a}, E) & \longmapsto \mathfrak{a} \cdot E
\end{align}
with cardinality equal to the class number $h(\mathcal{O})$.
\end{thm}
\begin{proof}
	See for example \cite[Chapter II]{Sil94}.
	
\end{proof}

\subsection{Proof of Theorem \ref{ordinarythm}}

Fix an ordinary elliptic curve $E$ over $\F_{q}$ as in the previous context. Assume $\End(E)=\mathcal{O}$. We compute the size of $I(q^{n}, E)$, which can be interpreted as the number of certain cyclic subgroups. We consider two kinds of subgroups; one is that the subgroups are the kernel of some horizontal isogenies, and the other is where the subgroups define non-horizontal isogenies.

\vspace{1em}

Let $\{\ell_{i}\}$, $ 1 \le i \le k$ be the set of prime divisors of $ \Delta_{\pi^{n}}$ which are unramified in $\mathcal{O}_{K}$. For each $i$, let $m_{i}$ be the \textit{maximal} integer such that $\ell^{2m} \mid \Delta_{\pi^{n}}$. By Lemma \ref{2m}, the $q^n$-Frobenius action is diagonalizable and the classification $(b.2)$ tells us that every $\ell^{j}$-subgroup, $1 \le j \le m$ is defined over $F_{q^{n}}$. For ordinary elliptic curves, $\Delta_{\pi^{n}} \ne 0$, so only finitely many primes divide $\Delta_{\pi^{n}}$. 

\begin{lem}\label{size ordinary}
We have $$N(q^{n}, E) \asymp (\prod_{i=1}^{k}\ell_{i}^{m_{i}})^{1-\epsilon}.$$ 
\end{lem}
\begin{proof}
	Denote by $\operatorname{Ell}_{q, \text{as/ds}}(\mathcal{O})$ the isomorphism classes of elliptic curves that admit an ascending/descending isogeny to $E$. Thus 
	$$N(q^{n}, E)= |\operatorname{Ell}_{q, \text{as/ds}}(\mathcal{O})| + |\operatorname{Ell}_{q}(\mathcal{O})|.$$
Since we have the assumptions on $n$, by Lemma \ref{2m}, $\pi^{n}$ is diagonalizable modulo any power of $\ell_{1}, \cdots, \ell_{k}$, and by Corollary \ref{does not grow with n}, the number of ramified prime-power isogenies does not grow with $n$. Thus we only have to consider isogenies from cases where the Frobenius action is diagonal, i.e., the number of isogenies grows with $n$. By Theorem \ref{classBound}, the number of horizontal isogenies $|\operatorname{Ell}_{q}(\mathcal{O})|=h(\mathcal{O})$ is a fixed number once we fix $E$.

The number of non-horizontal isogenies is roughly the number of cyclic subgroups of order less than or equal to $\ell_{1}^{m_{1}}\cdots\ell_{k}^{m_{k}}$, up to minus $h(\mathcal{O})$. This is because Lemma \ref{horizontal} implies that if $\Delta_{\pi^{n}} \equiv 0$ mod $\ell^{m}$, there is always a horizontal $\ell$-power isogeny, and Theorem \ref{classBound} tells us there are at most $h(\mathcal{O})$ horizontal isogenies come from this form. By Theorem \ref{Waterhouse}, For any cyclic subgroup $G$ of $\Z/\ell_{1}^{m_{1}}\cdots\ell_{k}^{m_{k}}\Z$, there are at most $(\ell_{1}^{m_{1}}\cdots\ell_{k}^{m_{k}})^{\epsilon}$ cyclic subgroups that give quotient curves isomorphic to $E/G$. Therefore
\begin{align}
N(q^{n}, E) &= \prod_{i=1}^{k}(\ell_{i}^{m_{i}}+\ell_{i}^{m_{i}-1}+\cdots+\ell)/(\ell_{1}^{m_{1}}\cdots\ell_{k}^{m_{k}})^{\epsilon} \\
	&\asymp (\ell_{1}^{m_{1}}\cdots\ell_{k}^{m_{k}})^{1-\epsilon}
\end{align}
\end{proof}

\begin{proof}[Proof of Theorem \ref{ordinarythm}]
	Lemma \ref{size ordinary} asserts that we can write $N(q^{n}, E)$ as a product of $\ell_{1}^{m_{1}}, \cdots, \ell_{k}^{m_{k}}$; on the other hand, for all but finitely many $n$, the product well approximates the square root of $\Delta_{\pi^n}$. Applying \cite[Lemma 5.4]{BF23}, we get 
$$\prod_{i=1}^{k}\ell_{i}^{m_i} \asymp \Delta^{\frac{1}{2}}_{\pi^n}\asymp q^{\frac{n}{2}(1+o(1))}.$$
The theorem follows.

\end{proof}

\subsection{Lower bound for the number of almost ordinary split surfaces}
\subsubsection{$q^n$ is a square}

 Let $A=E \times E_{ss}$ be an abelian surface defined over $\F_{q}$, with the assumption that $E$ is the same ordinary elliptic curve as in the previous section. The Frobenius $\pi_{A}^{n}$ acts on the $\ell$-adic Tate modules of $A$ by a conjugacy to \[
\begin{pmatrix}
  \pi^n
  & \rvline & 0   \\
\hline
  0 & \rvline &
  \begin{matrix}
  q^{n/2} & 0 \\
  0 &  q^{n/2}
  \end{matrix}
\end{pmatrix}
\]
where $\pi^n$ is the Frobenius of $E$ over $\F_{q^n}$. For the set of prime divisors of $\Delta_{\pi^{n}}$ which are unramified in $\mathcal{O}_{K}$ and positive integers $n$ such that $(n, \Delta_K) =1$, we want to count the number of inequivalent maximal isotropic planes defined over $\F_{q^{n}}$. By definition of $m_{i}$, for each $1 \le i \le k$, $\pi^{n}_{A}$ acts as a scalar on $A[\ell_{i}^{m_i}]$.	
	 Corollary \ref{product} together with the equality $$\prod_{i=1}^{k} \ell_{i}^{m_i} \asymp q^{\frac{n}{2}(1+o(1))}$$ indicate that for all but finitely many $n$, we have
     \begin{align}\label{lower bound split surfaces}
         N(q^{n}, A) &\gg ((\ell_{1}^{m_{1}} \cdots \ell_{k}^{m_{k}})^{2-\epsilon})=q^{n(1+o(1))}.
          \end{align}
\subsubsection{$q^n$ is not a square}
We will use the techniques of \cite{AH17} to obtain a lower bound of the same shape as in (\ref{lower bound split surfaces}). 
Let $E_{ss}'$ be an elliptic curve over $\F_{q^n}$ isogenous to $E_{ss}$ and let $E'$ be an elliptic curve over $\F_{q^n}$ isogenous to $E$. For every $n \ge 1$, recall that $t_n$ is the trace of $E$ over $\F_{q^n}$ and $\Delta_n=t_n^2-4q^n$. Take an integer $d$ that divides the difference of $t_n$ and the trace of $E_{ss}$ and let $$ \eta: E'[d] \to E_{ss}'[d]$$ be an anti-isometry with respect to the product of the Weil pairings. Due to the fact that if $\eta$ is an anti-isometry, then the graph $G$ of $\eta$ is a maximal isotropic subgroup of $(E' \times E_{ss}')[d]$ and the descending property of the canonical principal polarization on $E' \times E_{ss}'$ (\cite[Corollary to Theorem 2]{Mum}), the quotient $E' \times E_{ss}'/\operatorname{Graph}(\eta)$ is a principal polarized abelian surface which is split and almost ordinary. Frey and Kani \cite{FK91} show that every principally polarized abelian surface that is isogenous to a product of two elliptic curves arises in this way; furthermore,
if the surface is not isogenous to the square of an elliptic curve, then the $(E', E_{ss}', d,\eta)$ that give rise to $E' \times E_{ss}'/\operatorname{Graph}(\eta)$ are unique up to isomorphism and up to interchanging the triple $(E', E_{ss}', d, \eta)$ with $\left(E', E_{ss}', d, \eta^{-1}\right)$.

Using this fact, Achter and Howe \cite{AH17} converted the problem of counting such abelian surfaces over $\Fn$ into counting the size of the isogeny classes of $E$ and $E_{ss}$ over $\Fn$ and then for each divisor $d$ of the difference of traces count the number of such anti-isometries.  

We have already seen that the number of possible $E_{ss}'$ is $h(-p)$ and the number of possible $E'$ is the Kronecker class number of $\Delta_n$. Suppose $q^n$ is not a square. By \cite[Table 1]{AH17}, when $p \ne 2,3$ the trace of $E_{ss}$ is equal to $0$. So the difference of traces is $t_n$.

The Weil pairing yields a map $$
m: \operatorname{Isom}^1(E[n], F[n]) \rightarrow \operatorname{Aut} \boldsymbol{\mu}_n \cong(\mathbb{Z} / n \mathbb{Z})^{\times} .
$$
For every $i \in(\mathbb{Z} / n \mathbb{Z})^{\times}$, let $\operatorname{Isom}^i(E[n], F[n])$ be the set $m^{-1}(i)$, so that $$\operatorname{Isom}^1(E[n], F[n])$$ consists of the group scheme isomorphisms that respect the Weil pairing, and $\operatorname{Isom}^{-1}(E[n], F[n])$ consists of the anti-isometries from $E[n]$ to itself. For a prime power $\ell^k$ coprime to $q$, we say that two elliptic curves $E$ and $F$ on $\mathbb{F}_{q^n}$ are \textit{of the same symplectic type modulo $\ell^k$} if there is an isomorphism $E[\ell^k] \rightarrow F[\ell^k]$ of group schemes that respects the Weil pairing. In other words, $\operatorname{Isom}^1(E[\ell^k], F[\ell^k])$ is nonepmty.

Let $\ell$ be an odd prime divisor of $t_n$ and let $k$ be a positive integer such that $\ell^k \mid t_n$ and $\ell^{k+1} \nmid t_n.$ Since $E$ is ordinary, $\ell$ is coprime to $q$. We have $$0 \equiv t_n \not \equiv 4 q \bmod \ell^k$$ 
Therefore, $m: \operatorname{Aut}(E[\ell^k]) \rightarrow \text { Aut } \mu_{\ell^k}$ is surjective, so there is an isometry between $E[\ell^k]$ and $F[\ell^k]$ for any two elliptic curves in the isogeny classes of $E$ and $E_{ss}$. Moreover, denote by $t(E)$ the trace of an elliptic curve $E$, all elliptic curves $E / \mathbb{F}_{q^n}$ with $t(E) \equiv t_n \equiv 0 \bmod \ell$ are of the same symplectic type modulo $\ell^k$.

\begin{lem}
   Suppose $-1$ is a square modulo $t_n$. Then there are at least $o(\varphi(t_n))$ elements of $\operatorname{Isom}^{-1}(E'[\ell], E'_{ss}[\ell])$.
\end{lem}
\begin{proof}
    
We have shown that for any largest odd prime power $\ell^k$ that divides $t_n$ the set $\operatorname{Isom}^1(E'[\ell^k], E_{ss}'[\ell^k])$ is nonempty, i.e., there is an isometry $$\eta_{\ell^k}:  E'[\ell^k] \rightarrow  E_{ss}'[\ell^k].$$ Therefore, we have an isometry $\eta: E'[t_n] \rightarrow E'_{ss}[t_n]$. Let $b$ be an integer such that $b^2 \equiv-1 \bmod t_n$. Then $b \eta$ is an anti-isometry, so $$\operatorname{Isom}^{-1}(E'[t_n], E_{ss}'[t_n])$$ is nonempty. By \cite[Proposition 4.3 ]{AH17}, $$\#\operatorname{Aut}E'[t_n] \geq g^2\varphi(t_n)^2$$ where $g$ is either $1$ or $2$ because the divisors $d$ we consider here are coprime to $\Delta_n$. So  $\# \operatorname{Isom}^1(E'[t_n], E'[t_n])$ is at least $o(\varphi(t_n))$, and it follows that there are at least this many elements of $\operatorname{Isom}^{-1}(E'[t_n], E_{ss}'[t_n])$.
\end{proof}

Putting it all together, we get a lower bound of the form (up to logarithmic factors)
$$h(-p)H(\Delta_n)\varphi(t_n). $$

 Bombieri and Katz \cite{BK10} showed that the Frobenius trace of an ordinary elliptic curve over $\Fn$ has lower bound $$ |t_n| \ge (\frac{2}{\pi})\cdot q^{n/2-2^{37}\log (2n)}$$ and Hasse showed that $|t_n| \leq 2 \sqrt{q^n} .$ According to Achter and Howe \cite[Lemma 4.4]{AH17}, the Kronecker class number $H(\Delta_n)$ grows like $\sqrt{\left|t_n^2-4 q^n\right|} \sim 2 \sqrt{q^n}$ up to (at worst) logarithmic factors.
Therefore, the lower bound is in the desired form
$$O(q^{n(1+o(1))}).$$

\begin{appendices}
\section{Upper bound for almost ordinary split surfaces}

\subsection{$q^n$ is not a square}
Given $(E', E_{ss}', d\mid t_n)$ as in the previous context, Achter and Howe give an upper bound on the number of anti-isometries from $E'[d]$ to $E'_{ss}[d]$: 
\begin{lem}{\cite[Lemma 6.4]{AH17}} 
    Fix an elliptic curve $E / \mathbb{F}_{q^n}$ and a stratum $\mathcal{S}$ of elliptic curves over $\mathbb{F}_{q^n}$. Let $n$ be a positive integer. Suppose that either $\mathcal{S}$ is ordinary, or that a $(\mathcal{S})=0$ and $q$ is a nonsquare. If $\operatorname{Isom}^{-1}(E, \mathcal{S}, n)$ is nonempty then $\operatorname{gcd}\left(n, \mathfrak{f}_{\text {rel }}(E)\right)=\operatorname{gcd}\left(n, \mathfrak{f}_{\text {rel }}(\mathcal{S})\right)$, and we have

$$
\# \operatorname{Isom}^{-1}(E, \mathcal{S}, n) \leq 2 \psi(n) h\left(\mathcal{O}_{\mathcal{S}}\right) \operatorname{gcd}\left(n, \mathfrak{f}_{\mathrm{rel}}(E)\right) \operatorname{gcd}\left(n, \mathfrak{f}_{\mathrm{rel}}(\mathcal{S})\right)
$$

In particular, if $\mathfrak{f}_{\text {rel }}(E) \neq 0$, then

$$
\# \operatorname{Isom}^{-1}(E, \mathcal{S}, n) \leq 2 \psi(n) h\left(\mathcal{O}_{\mathcal{S}}\right) \mathfrak{f}_{\mathrm{rel}}(E) \mathfrak{f}_{\mathrm{rel}}(\mathcal{S}).
$$
Here $\mathfrak{f}_{\mathrm{rel}}(E)$ is the relative conductor of $E$ defined in \cite[Section 4]{AH17}.\end{lem}
In the present case, the relative conductors are either $1$ or $2$ due to the values of $d$ we consider are coprime to $\Delta_n.$ The lemma leads to an upper bound of the form
$$h(-p)H(\Delta_n)\varphi(t_n) $$
These estimates of size of $|t_n|$ and $H(\Delta_n)$ yields the upper bound $$O(q^{n(1+o(1))}).$$

\subsection{$q^n$ is a square}
Using the techniques in \cite{AH17}, we obtain a proof of the upper bound when $q^n$ is a square.
In this case, the Frobenius acts as a constant on $E_{ss}'[d]$ for $p \nmid d $ whose trace is either $2q^{n/2}$ or $-2q^{n/2}$. 

The difference $a(E')-a(E_{ss}')$ is either $r_n=t_n-2q^{n/2}$ or $s_n=t_n+2q^{n/2}$. Let $d$ be a divisor of $r_n$ or $s_n$. The set $\operatorname{Isom}^{-1}(E'[d], E_{ss}'[d])$ is nonempty if and only if the Frobenius action on $E'[d]$ is also a constant. This implies the discriminant of $\End(E')$ divides $r_n/d^2$ or $s_n/d^2.$  This inspires us to study the largest square factor of $r_n$ and $s_n.$

First we notice that 
Let $\Delta_1$ be the fundamental discriminant of $\mathcal{O}_K$. We can write
$$\Delta_n=\mathfrak{f}_n^2 \Delta_1$$
 where $ \mathfrak{f}_n$ is the conductor. Let $\mathfrak{f}_n=R_n S_n$ be the factorization of the conductor so that the square part of $r_n$ is $R_n^2$ and the square part of $s_n$ is $S_n^2$ up to a possible factor of $\operatorname{gcd}(r_n,s_n) $ which divides $4$.

By Lemma \ref{counting max iso planes}, the maximal number of possible anti-isometries from $E'[d]$ to $ E_{ss}'[d]$ is $O(d^3),$ so an upper bound for the total number of $(E', E_{ss}', d,\eta)$ is $$\sum_{d \mid R_n}d^3H(\Delta_n/d^2)+ \sum_{d \mid S_n}d^3H(\Delta_n/d^2)$$ and is $$(R^2_n +S_n^2)H(\Delta_n) $$ up to logarithmic factors. So the number we are counting is essentially $$(|r_n|+|s_n|)H(\Delta_n)=4q^{n/2}H(\Delta_n)$$ which yields the desired upper bound.

Finally, we remark that if the number of anti-isometries from $E'[d]$ to $ E_{ss}'[d]$ is on the order of $d^3$, then the argument above would give another proof of the lower bound in Theorem \ref{main}.
 \end{appendices}
\subsection*{Acknowledgements}
The author wishes to thank Ananth Shankar for encouraging her to think about this question and also for helpful conversations throughout this project. We are very grateful to Tonghai Yang for many helpful discussions.

Y.F. is partially supported by the NSF grant DMS$-2100436$.

\printbibliography
\addcontentsline{toc}{chapter}{Bibliography}

\end{document}